\documentclass{amsart}

\input xypic
 \xyoption{all}

\usepackage{bbm}
\usepackage{empheq}
\usepackage{color}
\usepackage[colorlinks,linkcolor=blue,anchorcolor=blue,citecolor=blue]{hyperref}
\usepackage{amssymb,amstext, amsbsy, amscd}
\usepackage[mathscr]{eucal}
\usepackage{times}
\usepackage{amsmath, amsthm, amsfonts,extarrows}

\usepackage[all,cmtip]{xy}

\newtheorem{thm}{Theorem}[section]
\newtheorem{prop}[thm]{Proposition}
\newtheorem{lemma}[thm]{Lemma}
\newtheorem{cor}[thm]{Corollary}

\newtheorem{remark}[thm]{Remark}

\newtheorem{conjecture}[thm]{Conjecture}

\newtheorem{prop-defn}[thm]{Proposition/Definition}

 \newcommand{\ba}{\begin{eqnarray}}
   \newcommand{\na}{\end{eqnarray}}
   \newcommand{\ban}{\begin{eqnarray*}}
   \newcommand{\nan}{\end{eqnarray*}}

\newcommand{\bC}{{\mathbb C}}

\newcommand{\bP}{{\mathbb P}}
\newcommand{\bQ}{{\mathbb Q}}
\newcommand{\bR}{{\mathbb R}}
\newcommand{\bZ}{{\mathbb Z}}

\newcommand{\cO}{{\mathcal O}}

  \newcommand{\<}{\langle}
  \renewcommand{\>}{\rangle}

\newtheorem*{conjO}{Conjecture $\cO$}

\newtheorem*{Galkin}{Galkin's lower bound conjecture }

\begin{document}{\allowdisplaybreaks[4]

\title{On the quantum cohomology of blow-ups of four-dimensional quadrics}

\author{Jianxun Hu }
\address{School of Mathematics, Sun Yat-sen University, Guangzhou 510275, P.R. China}

\email{stsjxhu@mail.sysu.edu.cn}
\thanks{ 
 }

\author{Hua-Zhong Ke}
\address{School of Mathematics, Sun Yat-sen University, Guangzhou 510275, P.R. China}
\email{kehuazh@mail.sysu.edu.cn}
\thanks{ 
 }

\author{Changzheng Li}
 \address{School of Mathematics, Sun Yat-sen University, Guangzhou 510275, P.R. China}
\email{lichangzh@mail.sysu.edu.cn}

\author{Lei Song }
\address{School of Mathematics, Sun Yat-sen University, Guangzhou 510275, P.R. China}
\email{songlei3@mail.sysu.edu.cn}
\thanks{ 
 }


\date{
      }




 \begin{abstract}
    We propose a conjecture relevant to Galkin's lower bound conjecture, and verify it 
    for the blow-ups of a four-dimensional quadric at a point or along a projective plane. We also show that Conjecture $\mathcal{O}$ holds in these two cases.
  \end{abstract}

\maketitle

\section{Introduction}

The (small) quantum cohomology $QH^*(X)=(H^*(X)\otimes_\mathbb{Q}\mathbb{Q}[\mathbf{q}], \star)$ of a Fano manifold  $X$ is a deformation of the classical cohomology ring $H^*(X)=H^*(X,\mathbb{Q})$, by incorporating three-pointed genus zero  Gromov-Witten invariants. The quantum multiplication by the first Chern class of $X$ induces a linear operator $\hat c_1$ on the finite-dimensional vector space $QH^\bullet(X):=QH^{\rm ev}(X)|_{\mathbf{q}=1}$. Attention has been drawn to the spectral properties of $\hat c_1$ in the past five years, mainly due to the following two conjectures. One is the  remarkable Conjecture $\mathcal{O}$ proposed by Galkin, Golyshev and Iritani \cite{GGI}, which concerns with  eigenvalues of $\hat c_1$ of modulus equal to the spectral  radius $\rho(\hat c_1)$ and underlies their Gamma conjecture I. The other one is an interesting conjecture proposed   by Galkin \cite{Gal}:
 \begin{Galkin}
   $\rho(\hat c_1)\geq \dim X+1$, with equality if and only if $X$ is isomorphic a projective space $\mathbb{P}^n$.
 \end{Galkin}
 \noindent Here and throughout this paper, we always consider the complex dimension. Conjecture $\mathcal{O}$  has been verified in various cases  (see \cite{HKLY} and references therein).  Galkin's lower bound conjecture was verified for few cases, including complex Grassmannians \cite{ESSSW}, Lagrangian and orthogonal Grassmannians \cite{ChHa}, and Fano complete intersections in projective spaces \cite{Ke}.

  Due to the lack of functorial properties of the quantum cohomology, individual studies of  relevant information on $QH^*(X)$ have to be taken   in general.   Nevertheless,
   when $\pi: X\to Y$ is a blow-up of Fano manifolds, {  the changing behavior of Gromov-Witten invariants under blowing up may be analyzed  via the absolute-relative correspondence} (see e.g. \cite{MaPa, HLR}). In some circumstances, the Gromov-Witten invariants for $X$ can be read off from that for $Y$ by explicit blow-up formulae \cite{Gath, Hu, Hu2, Lai}. Here we propose the following conjecture.
\begin{conjecture}\label{conj}
   Let $Y$ be a Fano manifold, and $Z$ be an irreducible smooth Fano subvariety of $Y$ with   {\upshape $\mbox{codim}Z\geq 2$}.
      If the blow-up $X=Bl_{Z}Y$ of $Y$ along $Z$ is Fano, then $\rho(\hat c_1(X))>\rho(\hat c_1(Y))$.
     \end{conjecture}
 \noindent We start with  $\mbox{codim}Z\geq 2$, since   the blow-up of $Y$ along a smooth irreducible divisor is isomorphic to $Y$ itself.
 We   even wish to remove the assumption of the smoothness of $Z$.

  When $Y=\bP^n$, the Fano submanifold  $Z$ can be taken as $\bP^r$ for any $0\leq r \leq n-2$. The blow-up $Bl_{\bP^r}\bP^n$ is a toric  Fano manifold, and the above conjecture is equivalent to Galkin's lower bound conjecture in this case. By using mirror symmetry \cite{GaIr} and analysing the mirror  superpotential, Conjecture $\mathcal{O}$ was verified for $Bl_{\bP^r}\bP^n$ in \cite{Yang}.  Numerical computations of the critical values of the same superpotential, which coincide with the eigenvalues of $\hat c_1$, provide
   a class of evidences for the above conjecture.
   When $Y$  is a del Pezzo surface given by  the blow-up of $\mathbb{P}^2$ at $r$ points in general position with $0\leq r\leq 7$, we can consider a point $Z$  in $Y$.
The blow-up $Bl_{Z}Y$ is again a del Pezzo surface, and
     Conjecture \ref{conj} can be directly verified by numerically solving the characteristic polynomial in \cite{BaMa}. This provides another main class of evidences, which support us to formulate  Conjecture \ref{conj}. Our conjecture gives a refined lower bound $\rho(\hat c_1(Y))$ of $\rho(\hat c_1(Bl_ZY))$ for blow-ups $Bl_ZY$ than Galkin's lower bound
      $\dim Bl_ZY+1=\dim Y+1$.

 Let $Q^4$ denote a  four-dimensional smooth quadric in $\mathbb{P}^5$. To further support our conjecture, we obtain the following main result of the present paper.

 \begin{thm}\label{mainthm}
    Conjecture $\mathcal{O}$,  Conjecture \ref{conj} and Galkin's lower bound conjecture all
 hold for the blow-ups $Bl_{\mathbb{P}^0} {Q}^4$ and $Bl_{\mathbb{P}^2} {Q}^4$.
\end{thm}
\noindent The quadric $Q^4$ can be interpreted as the complex Grassmannian $Gr(2, 4)$ via the Pl\"ucker embedding. The blow-ups in the above theorem, together with the blow-ups $Bl_{\bP^r}\bP^n$, can be uniformly treated as   Fano manifolds that arise from the blow-up of a complex Grassmannian $Gr(k, n)$ along a complex sub-Grassmannian.
We will study the general case systemically elsewhere.

For $r\in\{0, 2\}$, that Galkin's lower bound conjecture holds for $X_r=Bl_{\mathbb{P}^r} {Q}^4$   is an immediate consequence of Conjecture \ref{conj}  (see Corollary \ref{corGalkin}). We will prove the other two conjectures in Theorem \ref{thmQ4point} and  Theorem \ref{thmQ4P2} respectively.
The proofs rely on explicit computations of the quantum multiplication by the first Chern class $c_1(X_r)$. Thanks to the divisor axiom in Gromov-Witten theory, this reduces to the study of genus zero, two-point Gromov-Witten invariants with insertions in $H^*(X_r)$. To simplify the calculations, we explore the geometric structures of the Fano manifolds by noting that $X_0$ can also be realized the blow-up of $\bP^4$ along a two-dimensional quadric in a $3$-plane of $\bP^4$, and that
$X_2$ is a $\bP^2$-bundle over $\bP^2$. Part of the Gromov-Witten invariants are therefore easily obtained from the enumerative information of $Q^4$ by making use of the blow-up formula in \cite{Hu}. The proof of Conjecture $\mathcal{O}$ also relies on a good choice of bases of $H^*(X_r)$, so that Perron's theorem \cite{Perr} on positive matrices becomes applicable.
{We remark that all the three conjectures also hold for $Bl_{\bP^1}(Q^4)$, which is a natural case to study. While the geometric structure of $Bl_{\bP^1}(Q^4)$ is more complicated than that of $Bl_{\bP^0}(Q^4)$ and $Bl_{\bP^2}(Q^4)$, determination of two-point Gromov-Witten invariants of $Bl_{\bP^1}(Q^4)$ requires ad hoc arguments with WDVV equations involved. We therefore exclude it in this paper.}

The present paper is organized as follows. In section 2, we   review basic facts of quantum cohomology and introduce Conjecture $\mathcal{O}$. In section 3, we study the quantum
 cohomology of the blow-up of $Q^4$ at a point, and prove both Conjecture \ref{conj} and
  Conjecture $\mathcal{O}$ in this case. In section 4,  we study the quantum
 cohomology of the blow-up of $Q^4$ along a plane $\bP^2$, and prove both conjectures as well.
\section{Preliminaries}
 In this section, we review some basic facts, and  refer to \cite{CoKa} and \cite{GGI} for more details.

\subsection{Conjecture $\mathcal{O}$} Let $X$ be a Fano manifold, namely a compact complex manifold $X$ with positive  first Chern class $c_1(X)$. Let $\overline{\mathcal{M}}_{0, k}(X, \mathbf{d})$ denote the moduli stack of $k$-pointed genus $0$ stable maps $(f: C\to X; p_1, p_2, \cdots, p_k)$ of class $\mathbf{d}\in H_2(X,\mathbb{Z})$, which has a coarse moduli space $\overline{M}_{0, k}(X, \mathbf{d})$.
Its   virtual fundamental class $[\overline{M}_{0, k}(X, \mathbf{d})]^{\rm virt}$   is of complex degree $\dim_{\mathbb{C}}X+\int_{\mathbf{d}} c_1(X)+k-3$ in the Chow group $A_*(\overline{ {M}}_{0, k}(X, \mathbf{d}))$.
The k-pointed, genus 0 Gromov-Witten invariants of degree $\mathbf{d}$ for   $\gamma_1, \gamma_2, \cdots, \gamma_k\in H^*(X)=H^*(X, \mathbb{Q})$ is defined by
  $$\langle  \gamma_1,\gamma_2,\cdots, \gamma_k\rangle_{\mathbf{d}}^X:=\int_{[\overline{M}_{0, k}(X, \mathbf{d})]^{\rm virt}}  ev_1^*(\gamma_i)\cup  ev_2^*(\gamma_2)\cup \cdots\cup ev_k^*(\gamma_k).$$
  Here  $ev_i$ denotes the $i$-th  evaluation map.
  Set $m:=\mbox{rank}H_2(X,\mathbb{Z})$, take any    homogeneous basis $\{\phi_i\}_{i=1}^N$ of $H^*(X)$, and let $\{\phi^i\}$ denote  the dual basis of $H^*(X)$ that satisfy $(\phi_i, \phi^j)_X=\int_X \phi_i\cup \phi^j=\delta_{i, j}$ with respect to the Poincar\'e pairing.  The (small) quantum cohomogy ring $QH^*(X)=(H^*(X)\otimes_{\mathbb{Q}} \mathbb{Q}[q_1,\cdots, q_m], \star)$ is a deformation of the classical cohomology $H^*(X)$.  The quantum product of  $\alpha, \beta \in H^*(X)$ is given by
     $$\alpha\star \beta :=\sum_{\mathbf{d}\in H_2(X,\mathbb{Z})}\sum_{i=1}^N \langle \alpha,\beta,\phi_i\rangle_{\mathbf{d}}^X\phi^i q^{\mathbf{d}}.$$
     Here  $q^{\mathbf{d}}=\prod_{j=1}^mq_j^{d_j}$ for $\mathbf{d}=(d_1,\cdots, d_m)$ with a basis of effective curve classes of $H_2(X, \mathbb{Z})$ being fixed a prior. The quantum product  is a polynomial in $\mathbf{q}$, and   is independent of choices of the basis $\{\phi_i\}_i$.

  Consider the even part of the cohomology $H^\bullet(X):=H^{\rm even}(X)$ and
  let $QH^\bullet(X):=H^\bullet(X)\otimes \mathbb{Q}[\mathbf{q}]$.
  The first Chern class  $c_1(X)$ induces a linear operator  $\hat c_1=\hat c_1(X)$ by  the evaluation of the quantum product at $\mathbf{1}:=(1,\cdots, 1)$, namely defined by
      $$\hat c_1: QH^\bullet(X)\longrightarrow QH^\bullet(X); \, \beta\mapsto c_1(X)\star \beta|_{\mathbf{q}=\mathbf{1}}.$$
 Denote by $\rho=\rho(\hat c_1(X))$ the spectral radius of   $\hat c_1$, namely
    $$\rho:=\max\{|\lambda|~:~ \lambda\in \mbox{Spec}(\hat c_1)\}\quad\mbox{where}\quad \mbox{Spec}(\hat c_1):=\{\lambda~:~ \lambda\in \bC  \mbox{ is an eigenvalue of } \hat c_1\}.$$

 \begin{conjO}[Galkin-Golyshev-Iritani]Every Fano manifold  $X$ satisfies the following.
  \begin{enumerate}
   \item $\rho\in \mbox{Spec}(\hat c_1)$ and it is of multiplicity one.
   \item For any $\lambda\in \mbox{Spec}(\hat c_1)$ with $|\lambda|=\rho$, we have $\lambda^s=\rho^s$, where $s$ is the Fano index of $X$, namely $s=\max\{k\in \bZ~:~ {c_1(x)\over k}\in H^2(X,\mathbb{Z})\}$.
 \end{enumerate}
   \end{conjO}

\subsection{Quantum cohomology of $Q^4$}

Let $Q^4$ be a four-dimensional smooth quadric in $\bP^5$.
The quadric $Q^4$ can be interpreted as the image of the complex Grassmannian $Gr(2, 4)=\{V\leqslant \mathbb{C}^4\mid \dim V=2\}$ via the Pl\"ucker embedding
$Pl: Gr(2, 4)\hookrightarrow \bP^5$. Therefore the integral cohomology $H^*(Q^4, \mathbb{Z})\cong H^*(Gr(2, 4), \mathbb{Z})=\bigoplus_{2\geq a\geq b\geq 0}\mathbb{Z}\sigma^{(a, b)}$ has a $\mathbb{Z}$-basis of Schubert classes
 $\sigma^{(a, b)}$.

 Under the above identification,   the hyperplane class $H:=\sigma^{(1,0)} \in H^2(Q^4, \bZ)$ is the pullback of the hyperplane class in $H^2(\bP^5, \bZ)$. Moreover, $\wp:=\sigma^{(2,0)}-\sigma^{(1,1)}\in H^4(Q^4, \mathbb{Z})$ is a primitive class. Then the basis $\mathcal{B}$ of $H^*(Q^4, \bZ)$ of Schubert classes  can be written as
 	\[
	\mathcal{B}=\{\mathbf{1},H,\wp_+,\wp_-,\frac{1}{2}H^3,\frac{1}{2}H^4\},	
	\]
 in which $\wp_+:=\frac{1}{2}(H^2+\wp)=\sigma^{(2,0)}$,     $\wp_-:=\frac{1}{2}(H^2-\wp)=\sigma^{(1,1)}$,
  $\frac{1}{2}H^3=\sigma^{(2,1)}$ and $\frac{1}{2}H^4=\sigma^{(2,2)}$.
Geometrically, $Q^4$ has two rulings by planes. Then $\wp_+$ and $\wp_-$ are respectively given by their   cohomology classes, namely the Poincar\'e dual of their homology classes.     The cohomology class of a line is $\frac{1}{2}H^3$, and the cohomology class of a point is $\frac{1}{2}H^4$.  The dual basis $\mathcal{B}^\vee$ of $\mathcal{B}$ with respect to the Poincar\'e pairing is given by
	\[
	\mathcal{B}^\vee=\{\frac{1}{2}H^4,\frac{1}{2}H^3,\wp_+,\wp_-,H,\mathbf{1}\}.
	\]

Let $\ell\in H_2(Q^4,\bZ)$ be the homology class of a line. The non-zero three-point Gromov-Witten invariants with insertions in $\mathcal{B}$ can be directly read off from a table of quantum product of Schubert classes (see e.g. \cite[\S 8.2]{Miha}), and  are given as follows where  $\theta\in\{\wp_+,\wp_-\}$.
\begin{align*}
		\langle H^{0},H^{0},\frac{1}{2}H^{4}\rangle^{Q^4}_0&=1, &\langle H,\frac{1}{2}H^3, \frac{1}{2}H^{4}\rangle^{Q^4}_\ell&=1,\\
		\langle H^{0},H,\frac{1}{2}H^{3}\rangle^{Q^4}_0&=1,&\langle\wp_+,\wp_-, \frac{1}{2}H^{4}\rangle^{Q^4}_\ell&=1, \\
		\langle H^{0},\theta,\theta\rangle^{Q^4}_0&=1,&\langle\theta, \frac{1}{2}H^3, \frac{1}{2}H^{3}\rangle^{Q^4}_\ell&=1,\\
		\langle H,H,\theta,\rangle^{Q^4}_0&=1, &\langle \frac{1}{2}H^{4},\frac{1}{2}H^{4},\frac{1}{2}H^{4}\rangle_{2\ell}^{Q^4}&=1.
	\end{align*}
Closed formula for the spectral radius of  the linear operators induced by quantum multiplication by Schubert classes on $QH^\bullet(Gr(k,n))$
 has been given by Rietsch \cite{Riet} (see also \cite{ESSSW, LSYZ} for the formula). In particular for $Q^4=Gr(2,4)$, we have $c_1(Q^4)=4H=4\sigma^{(1,0)}$ and the following property.
\begin{prop}\label{proprhoQ}
   The spectral radius of $\hat c_1(Q^4)$ is equal to $4{\sin{2 \pi\over 4}\over \sin{\pi\over 4}}=4\sqrt{2}$.
\end{prop}
\begin{cor}\label{corGalkin}
   If Conjecture \ref{conj} holds for a Fano blow-up $Bl_Z(Q^4)$ of $Q^4$ along a subvariety $Z\subset Q^4$, so does Galkin's lower bound conjecture.
\end{cor}
\begin{proof}
   $\rho(\hat c_1(Bl_ZQ^4))>\rho(\hat c_1(Q^4)=4\sqrt{2}>4+1=\dim(Bl_ZQ^4)+1$.
\end{proof}
\section{Quantum cohomology of $Bl_{\rm \mathbb{P}^0}Q^4$}
In this section, we will study the quantum cohomology of the blow-up $X_0=Bl_{\rm \mathbb{P}^0}Q^4$ of $Q^4$ at one point, and verify Conjecture $\mathcal{O}$ and Conjecture \ref{conj} for $X_0$.

\subsection{Geometric construction}

We will review a geometric construction of $X_0$ (see e.g. \cite[\S 22]{Harr}), which can be interpreted as both the blow-up of $Q^4$ at a point and the blow-up of $\mathbb{P}^4$ along a two-dimensional quadric $Q_0^2$ in a 3-plane of $\mathbb{P}^4$. Both interpretations will be used to compute the relevant 2-pointed genus zero Gromov-Witten invariants of $X_0$.

 Fix $x\in Q^4$ and a $4$-plane $\bP^4_0$ in $\bP^5$, such that $x\notin \bP^4_0$. The rational map $f: \bP^5\dashrightarrow \bP^4_0$, given by the projection from $x$ to $\bP^4_0$, defines a graph $\Gamma_f\subset\bP^5\times \bP^4_0$ by the Zariski closure of $\{(y, f(y))\mid y\in \bP^5\setminus\{x\}\}$.
 There are  two natural projections $\Gamma_f\xrightarrow{p_1}\bP^5$ and $\Gamma_f\xrightarrow{p_2}\bP^4_0$. The morphism $p_1$ is the blow-up of $\bP^5$ at $x$; the morphism $p_2$ endows the graph  with the $\bP^1$-bundle structure $\Gamma_f=\bP_{\bP^4_0}(\cO_{\bP^4_0}(1)\oplus\cO_{\bP^4_0})$, which can be viewed as the projective compactification of
\(
\bP^5\setminus\{x\}=N_{\bP^4_0|\bP^5}\cong\cO_{\bP^4_0}(1).
\)
\begin{lemma}\label{lemS}
	For any smooth subvariety $S$ in $\bP^4_0$, we have
	\[
	p_2^{-1}(S)\cong\bP_S(\cO_S(1)\oplus\cO_S),
	\]
	where $\cO_S(1)$ is the pullback of the hyperplane line bundle $\cO_{\bP^4_0}(1)$.
\end{lemma}
\begin{proof}
  $p_1(p_2^{-1}(S))$ is the cone $\overline{xS}$ in $\bP^5$ over $S$ with vertex $x$.  The statement follows by noting  that
   the subvariety $p_2^{-1}(S)$ of $\Gamma_f$ is the projective compactification of
\[
\overline{xS}\setminus\{x\}=N_{\bP^4_0|\bP^5}|_S\cong\cO_{\bP_0^4}(1)|_{S}=\cO_S(1).
\]
\end{proof}

Let  $X_0\subset\bP^5\times\bP_0^1$ be the strict transform of $Q^4$, that is, $X_0$ is the Zariski closure of $p_1^{-1}(Q^4\setminus\{x\})$.
Consider the restrictions
$$\pi_1=p_1|_{X_0}: X_0\to Q^4,\quad \pi_2=p_2|_{X_0}: X_0\to \bP^4.$$
\noindent The morphism $\pi_1$ is the blow-up of $Q^4$ at $x$, and its exceptional divisor is given by
\[
E_0:=X_0\cap D_\infty,\]
where $D_\infty:=\bP_{\bP^4_0}(\cO_{\bP^4_0}(1)\oplus\{0\})$ is the exceptional divisor of the blow-up $p_1$.

The tangent hyperplane $H_x$ of $Q^4$ at $x$, consisting of the tangent lines in $T_x\bP^5|_{Q^4}$, is naturally viewed as   a $4$-plane in $\bP^5$.
Noting that $E_0$ is a $3$-plane in $D_\infty\cong\bP^4$, we have
\begin{lemma}
	$\pi_2(E_0)$ is  a $3$-plane in $\bP^4_0$. Moreover, $H_x\cap \bP^4_0=\pi_2(E_0)$.
\end{lemma}
\begin{proof}
The first statement follows directly from the construction of $X_0$. Indeed, for $w\in E_0$, let $l_w$ be the tangent line of $Q^4$ at $x$ with tangent direction given by $w$. Note that $l_w$ is a line in $\bP^5$ not contained in $\bP^4_0$, and $\pi_2(w)$ is the unique intersection point of $l_w$ and  $\bP^4_0$.

On one hand, for $y\in H_x\cap \bP^4_0$, the line $\overline{xy}$ is a tangent line of $Q^4$ at $x$, and $\overline{xy}\cap \bP^4_0=\{y\}$, which implies that $y\in\pi_2(E_0)$. On the other hand, for $w\in E_0$,   we have $l_w\subset H_x$ and  hence $\pi_2(w)\in H_x\cap \bP^4_0$. Therefore the second statement follows.
\end{proof}

The quadric $3$-fold $H_x\cap Q^4$ in $H_x$ 
 is a cone with vertex $x$ over
\[Q^2_0:=H_x\cap Q^4\cap\bP^4_0=Q^4\cap\pi_2(E_0).\]
The surface $Q^2_0$ is a smooth quadric in the $3$-plane $\pi_2(E_0)$.  By the definition of $Q^2_0$, for any $y\in Q^2_0$, the line $\overline{xy}$ is contained in $Q^4$.
This implies $\pi_2^{-1}(Q^2_0)=p_2^{-1}(Q^2_0)$. Thus by   Lemma \ref{lemS},   we have
\[
E':=\pi_2^{-1}(Q^2_0)\cong\bP_{Q_0^2}(\cO_{Q_0^2}(1)\oplus\cO_{Q^2_0}).
\]

The morphism  $\pi_2: X_0\to\bP^4_0$ is actually the blow-up of $\bP^4_0$ along $Q^2_0$ \cite{Harr}, with the exceptional divisor $E'$. Under the above identification, the intersection $E'\cap E_0$ is given by
\[
Q^2_\infty:=\bP_{Q_0^2}(\cO_{Q_0^2}(1)\oplus\{0\}),
\]
which is a smooth quadric in $E_0\cong\bP^3$.

\subsection{Classical cohomology} Arising from the blow-up $\pi_1: X_0\to Q^4$, the  integral cohomology   $H^*(X_0, \mathbb{Z})$ has a basis $\mathcal{B}_0$ together with its dual basis $\mathcal{B}_0^\vee$ with respect to the Poincar\'e pairing, coming from
  $\pi_1^*(H^*(Q^4, \mathbb{Z}))$ and the exceptional classes. Let $P_0$ be a plane in the exceptional divisor $E_0\cong \bP^3$, and   $L_0$ be  a line in $E_0$. Precisely, we have
\begin{align*}
   \mathcal{B}_0&:=\{\mathbf{1},H,\wp_+,\wp_-,\frac12H^3,\frac12H^4,E_0,P_0,L_0\},\\
   \mathcal{B}_0^\vee&:=\{\frac{1}{2}H^4,\frac{1}{2}H^3,\wp_+,\wp_-,H,\mathbf{1},
-L_0,-P_0,-E_0.\}.
\end{align*}
 Here by abuse of notation, we have simply denoted the pullback of a class $\alpha\in H^*(Q^4, \mathbb{Z})$ by $\alpha$, and simply denote    the cohomology class of    a subvariety $S$ of $X_0$ by $S$ as well.

By direct calculations, the non-trivial products on $\mathcal B_0$ are given as follows, where $\theta\in\{\wp_+, \wp_-\}$.
\begin{align*}
   H\cup H&=\wp_++\wp_-,& E_0\cup E_0&=-P_0,\\
   H\cup\theta&=\frac12H^3,&  E_0\cup P_0&=-L_0,\\
   H\cup\frac12H^3&= \frac12H^4,& E_0\cup L_0&=-\frac12H^4,\\
   \theta\cup\theta&=\frac{1}{2}H^4,&P_0\cup P_0&=-\frac12H^4.
\end{align*}

Note that $Q^2_0\cong \bP^1\times \bP^1$ has two rulings by lines. We fix two lines $l_+,l_-\subset Q^2_0$ in different rulings, and denote by $S'_\pm$ the class of $\pi_2^{-1}(l_\pm)$. Let $F'$ be the class of a fiber of $E'\cong\bP_{Q_0^2}(\cO_{Q_0^2}(1)\oplus\cO_{Q^2_0})$. Let $H'$ be the pullback of the hyperplane class in $\bP^4_0$ of $\pi_2^*$. Arising from the blow-up $\pi_2: X_0\to \bP^4_0$,     $H^*(X_0, \mathbb{Z})$ has a basis $\mathcal{B}_0'$ together with its dual basis $\mathcal{B}_0'^\vee$,
given by
\begin{align*}
   \mathcal{B}_0'&=\{\mathbf1,H',H'^2,H'^3,H'^4,E',S'_+,S'_-,F'\},\\
    \mathcal{B}_0'^\vee&=\{H'^4,H'^3,H'^2,H',\mathbf1,-F',-S'_-,-S'_+,-E'\}.
\end{align*}
 \subsection{Comparison of two bases of $H^*(X_0)$}
In this subsection, we  compare the two bases $\mathcal{B}_0$ and $\mathcal{B}'_0$ of $H^*(X_0, \mathbb{Z})$. We observe that $\frac12H^4=H'^4$.

\subsubsection{Comparison in $H^2(X_0)$ and $H^6(X_0)$}

Let $e\in H_2(X_0,\bZ)$ be the homology class of a line in $E_0\cong\bP^3$. Denote also by $\ell$ the Poincar\'e dual of $\frac{1}{2}H^3$, whose image via $\pi_1{}_*$ is
 the  class of a line in  $Q^4$. Then $e$ and $\ell$ form a basis of $H_2(X_0,\bZ)$, and we have
\[
\left\{\begin{array}{ccc}
H.\ell&=&1,\\
H.e&=&0,
\end{array}\right.\mbox{ and }
\left\{\begin{array}{ccc}
E_0.\ell&=&0,\\
E_0.e&=&-1.
\end{array}\right.
\]
Moreover, the classes $e$ and $\ell-e$ generate the Mori cone of $X_0$, and $X_0\xrightarrow{\pi_1}Q^4$ is the contraction associated to the extremal ray $\bR_{\geq0}e$.

Let $e'\in H_2(X_0,\bZ)$ be the homology class of a fiber of $E'$, 
and $\ell'\in H_2(X_0,\bZ)$ be the Poincar\'e dual of
 $H'^3$. 
 Then $e'$ and $\ell'$ form another basis of $H_2(X_0,\bZ)$, and we have
\[
\left\{\begin{array}{ccc}
H'.\ell'&=&1,\\
H'.e'&=&0,
\end{array}\right.\mbox{ and }
\left\{\begin{array}{ccc}
E'.\ell&=&0,\\
E'.e&=&-1.
\end{array}\right.
\]

\begin{lemma}
$e'=\ell-e$.
\end{lemma}
\begin{proof}
Note that the Mori cone of $X_0$ has only two extremal rays $\bR_{\geq0}e$ and $\bR_{\geq0}(\ell-e)$. The contraction $X_0\xrightarrow{\pi_1}Q^4$ corresponds to the extremal ray $\bR_{\geq0}e$, and the contraction $X_0\xrightarrow{\pi_2}\bP^4_0$ corresponds to the extremal ray $\bR_{\geq0}e'$. So $e\neq e'$, and hence $e'=\ell-e$.
\end{proof}

\begin{lemma}
$e=\ell'-2e'$.
\end{lemma}
\begin{proof}
Write $e=a\ell'+be'$ for some $a,b\in\bC$. Since a line in $E_0\cong\bP^3$ is mapped via $\pi_2$ to a line in the $3$-plane $\pi_2(E_0)\subset\bP^4_0$, it follows that $a=1$. Moreover, since a line in $E_0$ in general position meets $Q^2_\infty=E'\cap E_0$ at two points, it follows that $E'.e=-2$, which implies that $b=-2$.
\end{proof}

So we see that
\[
\left\{
\begin{array}{ccc}
\ell'&=&2\ell-e,\\
e'&=&\ell-e,
\end{array}
\right.\textrm{ and }
\left\{
\begin{array}{ccc}
\ell&=&\ell'-e',\\
e&=&\ell'-2e'.
\end{array}
\right.
\]
By considering Poincar\'e duals, we obtain
\[
\left\{
\begin{array}{ccc}
H'^3&=&H^3-L_0,\\
F'&=&\frac{1}{2}H^3-L_0,
\end{array}
\right.\textrm{ and }
\left\{
\begin{array}{ccc}
\frac12H^3&=&H'^3-F',\\
L_0&=&H'^3-2F'.
\end{array}
\right.
\]
By considering Poincar\'e pairings, we obtain
\[
\left\{
\begin{array}{ccc}
H'&=&H-E_0,\\
E'&=&H-2E_0,
\end{array}
\right.\textrm{ and }
\left\{
\begin{array}{ccc}
H&=&2H'-E',\\
E_0&=&H'-E'.
\end{array}
\right.
\]

\subsubsection{Comparison in $H^4(X_0)$}

Note  $\pi_2^{-1}(Q^2_0)=p_2^{-1}(Q^2_0)$. Since $l_\pm\subset Q^2_0$, it follows that
\[
\pi_1(\pi_2^{-1}(l_\pm))=\pi_1(p_2^{-1}(l_\pm))=p_1(p_2^{-1}(l_\pm)).
\]
So $\pi_1(\pi_2^{-1}(l_\pm))$ is the cone over the line $l_\pm$ with vertex $x$, and hence it is a $2$-plane. Moreover, for any $y\in l_\pm$, the line $\overline{xy}$ is contained in $Q^4$. So the $2$-plane $\pi_1(\pi_2^{-1}(l_\pm))$ is contained in $Q^4$, and the intersection
\[
\pi_1(\pi_2^{-1}(l_+))\cap\pi_1(\pi_2^{-1}(l_-))
\]
is the line passing through $x$ and $l_+\cap l_-$. So we see that $\pi_1(\pi_2^{-1}(l_+))$ and $\pi_1(\pi_2^{-1}(l_-))$ belong to different rulings by $2$-planes of $Q^4$. Rename $l_+$ and $l_-$ if necessary so that $[\pi_1(\pi_2^{-1}(l_\pm))]=\wp_\pm$,

\begin{lemma}
$S'_\pm=\wp_\pm-P_0$.
\end{lemma}
\begin{proof}
We have
\begin{align*}
S'_\pm=\wp_\pm+a_\pm P_0,\mbox{ for some }a_\pm\in\bQ.
\end{align*}
Since $\pi_2^{-1}(l_\pm)\cap E_0$ is a line in $E_0\cong\bP^3$, and $P_0$ is the class of a $2$-plane in $E_0$, it follows that
\begin{align}\label{eq-lineplaneintersectioninE_0}
S'_\pm.P_0=1
\end{align}
Since $P_0.P_0=-1$, it follows that $a_\pm=-1$.
\end{proof}

\begin{lemma}
$P_0=H'^2-S'_+-S'_-$.
\end{lemma}
\begin{proof}
Since the image under $\pi_2$ of a $2$-plane in $E_0\cong\bP^3$ is a $2$-plane in $\pi_2(E_0)\subset\bP^4_0$, it follows that
\[
P_0=H'^2+b_+S'_++b_-S'_-,\mbox{ for some }b_\pm\in\bQ.
\]
Since $S'_\pm.S'_\pm=0$ and $S'_\pm.S'_\mp=-1$, it follows from \eqref{eq-lineplaneintersectioninE_0} that $b_\pm=-1$.
\end{proof}

So we have
\[
\left\{\begin{array}{ccc}
S'_\pm&=&\wp_\pm-P_0,\\
H'^2&=&\wp_++\wp_--P_0,
\end{array}\right.\textrm{ and  }\left\{\begin{array}{ccc}
\wp_\pm&=&H'^2-S'_\mp,\\
P_0&=&H'^2-S'_+-S'_-.
\end{array}\right.
\]

\subsection{Two-point invariants}

\begin{lemma}\label{lemma-anotherblowupofQ4atonepoint}
	The non-zero, degree-$ke'$ for $k\geq1$, two-point invariants with insertions in $\mathcal{B}_0'$ are
	\[
	\<S'_+,S'_-\>^{X_0}_{e'}=\<E',F'\>^{X_0}_{e'}=1.
	\]
\end{lemma}
\begin{proof}
	Let $\iota:E'\hookrightarrow X_0$ be the natural embedding, and we also denote by $\iota$ the induced embedding of moduli spaces of stable maps:
	\[
	\iota:	\overline M_{0,2}(E',ke')\hookrightarrow\overline M_{0,2}(X_0,ke').
	\]
	Note that the blow-down morphism $X_0\xrightarrow{\pi_2}\bP^4_0$ is the contraction corresponding to the extremal ray $\bR_{\geq0}e'$. So we have
	\[
	\iota\Big(\overline M_{0,2}(E',ke')\Big)=\overline M_{0,2}(X_0,ke').
	\]
	Consider the universal diagram
	\[
	\begin{CD}
		\overline M_{0,3}(E',ke')@>ev_3>>E'\\
		@V \pi VV  \\
		\overline M_{0,2}(E',ke')
	\end{CD}
	\]
	and let $R:=R^1\pi_*ev^*_3N_{E'|X_0}$, where $\pi$ denotes the  natural morphism by forgetting the third marking point. From the construction of virtual fundamental classes, we have
	\[
	\iota_*\big(\mathbf{e}(R)\cap[\overline M_{0,2}(E',ke')]^{\rm virt}\big)=[\overline M_{0,2}(X_0,ke')]^{\rm virt},
	\]
	where $\mathbf{e}(R)$ is the Euler class of $R$. So for $B_1,B_2\in\mathcal B'_0$, we see that
	\[
	\<B_1,B_2\>_{ke'}^{X_0}=\int_{[\overline{M}_{0, 2}(E',ke')]^{\rm virt}}ev_1^*\iota^*B_1\cup ev_2^*\iota^*B_2\cup\mathbf{e}(R).
	\]
	For $i=1,2$, consider the morphism $f_i:\overline M_{0,2}(E',ke')\to Q^2_0$ defined by the composition
	\[
	f_i:=\pi_2\circ ev_i.
	\]
	We see that $f_1=f_2$, and we let $f:=f_1=f_2$. Then we have
	\[
	\<B_1,B_2\>_{ke'}^{X_0}=\int_{Q^2_0}PD\Big(f_*\big(ev_1^*\iota^*B_1\cup ev_2^*\iota^*B_2\cup\mathbf{e}(R)\cap[\overline M_{0,2}(E',ke')]^{\rm virt}\big)\Big).
	\]
	Let $L_\pm\in H^2(Q^2_0)$ be the class of $l_\pm$, and $P\in H^4(Q^2_0)$ be the class of a point. Then $\mathbf{1}, L_+, L_-, P$ form a basis $\mathcal B''$ of $H^*(Q^2_0)$. Set $\xi:=c_1(N_{E'|X_0})\in H^2(E')$. We have
 \begin{align*}
	\iota^*\mathbf{1}&=\mathbf{1}, &
	\iota^*H'&=\pi_2^*(L_++L_-), &
	\iota^*H'^2&=\pi_2^*(2P), &
	\iota^*H'^3&=0,\\
	\iota^*H'^4&=0, &
	\iota^*E'&=\xi, &
	\iota^*S'_\pm&=\pi_2^*L_\pm\cup\xi, &
	\iota^*F'&=\pi_2^*P\cup\xi.
	\end{align*}
 We see that for $i=1,2$, $\iota^*B_i$ has the form
\[
\iota^*B_i=\pi_2^*\underline{B_i}\cup\xi^{\alpha(i)},\quad \underline{B_i}\in\mathcal B'',\alpha(i)\in\{0,1\},
\]
which implies that
\[
ev_i^*\iota^*B_i=f^*\underline{B_i}\cup ev_i^*\xi^{\alpha(i)}.
\]
So we use the projection formula to get
\[
\<B_1,B_2\>_{ke'}^{X_0}=\int_{Q^2_0}\underline{B_1}\cup\underline{B_2}\cup PD\Big(f_*\big(ev_1^*\xi^{\alpha(1)}\cup ev_2^*\xi^{\alpha(2)}\cup\mathbf{e}(R)\cap[\overline M_{0,2}(E',ke')]^{\rm virt}\big)\Big).
\]
Assume that $\<B_1,B_2\>_{ke'}^{X_0}\neq0$. Then the above formula gives
\[
\deg\underline{B_1}+\deg\underline{B_2}\leq4\mbox{ and }\underline{B_1}\cup\underline{B_2}\neq0.
\]
We use the dimension constraint to get
\[
(4-3)+3k+2=\Big(\deg\underline{B_1}+\deg\underline{B_2}\Big)+\Big(\alpha(1)+\alpha(2)\Big)\leq4+2=6,
\]
which implies that $k=1$. We use the dimension constraint again to get
\[
\deg\underline{B_1}+\deg\underline{B_2}=4\mbox{ and }\alpha(1)=\alpha(2)=1.
\]
From the condition $\underline{B_1}\cup\underline{B_2}\neq0$, we only need to consider the cases
\[
(B_1,B_2)=(E',F')\mbox{ and }(S'_+,S'_-).
\]
Since $k=1$, it follows from $H^1(\bP^1,\cO(-1))=0$ that $\mathbf{e}(R)=0$. So we obtain
\[
\<E',F'\>_{e'}^{X_0}=\<\xi,\pi_2^*P\cup\xi\>_{e'}^{E'}\mbox{ and }\<S'_+,S'_-\>_{e'}^{X_0}=\<\pi_2^*L_+\cup\xi,\pi_2^*L_-\cup\xi\>_{e'}^{E'}.
\]
Both of the Gromov-Witten invariants of $E'$ are equal to one, since there is a unique line contained in a fiber of $E'$.
\end{proof}

\begin{lemma}\label{lemnonzero1}
	A curve class in $H_2(X_0, \mathbb{Z})$ admits non-zero two-point invariants only if it belongs to
	$\{
	\ell-e,e,\ell,2\ell-e
	\}$.
\end{lemma}
\begin{proof}
An effective curve class has the form
\[
a(\ell-e)+be,\quad a,b\in\bZ_{\geq0}.
\]
Notice $c_1(X_0)=4H-3E_0$ (see e.g. \cite[Chapter 4, Section 6]{GrHa})). Using the dimension constraint, we see that $a(\ell-e)+be$ admits non-zero two-point invariants only if
\[
a+3b\leq5.
\]
Note that $(a,b)\neq(0,0)$ since the space $\overline M_{0,2}(X_0,0)$ is empty. Note that $\ell-e=e'$. So from Lemma \ref{lemma-anotherblowupofQ4atonepoint}, we can exclude the cases $k(\ell-e)$ for $2\leq k\leq5$.
	\end{proof}

\begin{lemma}\label{lemnonzero2}
	The non-zero, degree-$(\ell-e)$, two-point invariants with insertions in $\mathcal{B}_0$ are
	\[
	\<P_0,P_0\>_{\ell-e}^{X_0}=\<H,L_0\>_{\ell-e}^{X_0}=\<E_0,L_0\>_{\ell-e}^{X_0}=2
	\]
	and
	\[
	\<P_0,\wp_\pm\>_{\ell-e}^X=\<\wp_+,\wp_-\>_{\ell-e}^X=\<H,\frac12H^3\>_{\ell-e}^X=\<E_0,\frac12H^3\>_{\ell-e}^X=1.
	\]
\end{lemma}
\begin{proof}
	From Lemma \ref{lemma-anotherblowupofQ4atonepoint}, the only non-zero, degree-$e'$, two-point invariants with insertions in $\mathcal{B}_{0}'$ are
	\[
	\<S'_+,S'_-\>^{X_0}_{e'}=\<E',F'\>^{X_0}_{e'}=1.
	\]
	Note that $e'=\ell-e$, and we can use this to determine invariants with insertions in $\mathcal{B}_0$. For example,
	\[
	\<\wp_+,\wp_-\>_{\ell-e}^{X_0}=\<H'^2-S'_-,H'^2-S'_+\>_{e'}^{X_0}=1.
	\]
	We leave the rest cases to interested readers.
\end{proof}

\begin{lemma}\label{lemnonzero3}
	The only non-zero, degree-$e$, two-point invariant with insertions in $\mathcal{B}_0$ is \[\<L_0,L_0\>_e^{X_0}=1.\]
\end{lemma}
\begin{proof}
The proof is similar to that of Lemma \ref{lemma-anotherblowupofQ4atonepoint}, and we leave it to interested readers.
\end{proof}

\begin{lemma}\label{lemnonzero4}
	The only non-zero, degree-$\ell$, two-point invariants with insertions in $\mathcal{B}_0$ is
	\[
	\<\frac12H^3,\frac12H^4\>_{\ell}^{X_0}=1.
	\]
\end{lemma}
\begin{proof}
By the dimension constraint, we only need to consider the cases
\[
\<\frac12H^3,\frac12H^4\>_{\ell}^{X_0}\textrm{ and }\<L_0,\frac12H^4\>_{\ell}^{X_0}.
\]
For the former case, we use the blow-up formula \cite[Theorem 1.2]{Hu}. For the latter case, observe that for $y\in Q^4$, there is a line in $Q^4$ passing through $x$ and $y$ iff $y\in H_x\cap Q^4$. As a consequence, for $y\in Q^4$ in general position, there is no line in $Q^4$ passing through $x$ and $y$. So we have $\<L_0,\frac12H^4\>_{\ell}^{X_0}=0$.
\end{proof}

\begin{lemma}\label{lemnonzero5}
	The only non-zero, degree-$(2\ell-e)$, two-point invariants with insertions in $\mathcal{B}_0$ is
	\[
	\<\frac12H^4,\frac12H^4\>_{2\ell-e}^{X_0}=1.
	\]
\end{lemma}
\begin{proof}
Use the dimension constraint and blow-up formula \cite[Theorem 1.4]{Hu}.
\end{proof}

In general, there are very few tools to study the eigenvalues of linear operators on a vector space. One remarkable tool is the Frobenius-Perron theory on irreducible nonnegative matrices \cite{Perr, Frob} (see e.g. \cite{Minc}), provided that a \textit{good} basis of the vector space could be found.
Now we consider the operator $\hat c_1(X_0)$, where we recall $c_1(X_0)=4H-3E_0$.
\begin{thm}\label{thmQ4point}
    Both  Conjecture \ref{conj} and
  Conjecture $\mathcal{O}$ hold for the blow-up  $X_0=Bl_{\mathbb{P}^0} {Q}^4$.
\end{thm}
\begin{proof}
 Instead of the bases $\mathcal{B}_0$ and $\mathcal{B}_0'$, we consider another basis of  $H^*(X_0,\mathbb{Z})$ given by
$$\hat{\mathcal{B}}_0:=\{\mathbf{1}, H, H-E_0, \wp_+, \wp_-, \wp_++\wp_--P_0, {1\over 2}H^3, {1\over 2}H^3-L_0, {1\over 2}H^4\}. $$
By Lemmas \ref{lemnonzero1}, \ref{lemnonzero2}, \ref{lemnonzero3}, \ref{lemnonzero4} and \ref{lemnonzero5}, we have
\begin{align*}
  c_1(X_0)\star H|_{\mathbf{q}=\mathbf{1}} &= c_1(X_0)\cup H+\langle c_1(X_0), H, -L_0\rangle_{\ell-e}^{X_0}E_0+ \langle c_1(X_0), H, {1\over 2}H^3\rangle_{\ell-e}^{X_0}H\\
   &=4\wp_++4\wp_-+(c_1(X_0).(\ell-e))(\langle  H, -L_0\rangle_{\ell-e}^{X_0}E_0+ \langle H, {1\over 2}H^3\rangle_{\ell-e}^{X_0}H)\\
   &=4\wp_++4\wp_--H+2(H-E).
  \end{align*}
  Here the second equality follows from  the divisor axiom for Gromov-Witten invariants (see e.g. \cite{CoKa}) and the computation $c_1(X_0).(\ell-e)=(4H-3E_0).(\ell-e)=1$. Similarly, we can
  calculate  the product  $c_1(X_0)\star \alpha|_{\mathbf{q}=\mathbf{1}}$ for all $\alpha\in \hat{\mathcal{B}}_0$. As a result, we obtain
  the associated matrix $M$ in  $\hat c_1(X_0)\hat{\mathcal{B}}_0=\hat{\mathcal{B}}_0 M$ with
 $$M=\left(
    \begin{array}{ccccccccc}
0& 0& 0& 0& 0& 0& 4& 4& 5\\1& -1& 0& 0& 0& 0& 0& 3& 4\\3& 2& 0& 0& 0& 0& 0& -3& 0\\0& 4& 1& -1& 0& 0& 0& 0& 0\\0& 4& 1& 0& -1& 0& 0& 0& 0\\0& 0& 3& 1& 1& 0& 0& 0& 0\\0& 0& 0& 4& 4& 5& 0& 0& 0\\0& 0& 0& 0& 0& 3& 1& -1& 0\\0& 0& 0& 0& 0& 0& 4& 1& 0
    \end{array}
  \right).
$$

By direct calculations, the power $M^{17}$ is a positive matrix. Thus by Perron's theorem on positive matrices \cite{Perr}, the spectral radius $\rho(M^{17})$ is a simple eigenvalue of $M^{17}$, and the modulus of any other eigenvalue of $M^{17}$ is strictly less than $\rho(M^{17})$. Since $M$ is a real matrix, if $\lambda$ is an eigenvalue of
$M$, so is the conjugate $\bar \lambda$. It follows that    $\rho(M)=(\rho(M^{17}))^{1\over 17}$ is a simple eigenvalue of $M$, and the modulus of any other eigenvalue of $M$ is strictly less than $(\rho(M^{17}))^{1\over 17}$. Hence, Conjecture $\mathcal{O}$ holds for $X_0$. (Here we notice the Fano index of $X_0$ is equal to $1$, since
 $c_1(X_0).(\ell-e)=1$.)

Notice that $\det(\lambda I_{9}-M)|_{\lambda=4\sqrt{2}}<0$. Hence, there exists a real eigenvalue $\lambda_0$ in the interval $(4\sqrt{2}, +\infty)$. Hence, by Proposition \ref{proprhoQ}, we have
 $$\rho(\hat c_1(X_0))=\rho(M)\geq \lambda_0>4\sqrt{2}=\rho(\hat c_1(Q^4)).$$
That is, Conjecture \ref{conj} holds for $Y=Q^4$ and $Z=\mathbb{P}^0$.
\end{proof}
\begin{remark}
To prove Conjecture $\mathcal{O}$, one can also apply the generalized Frobenius-Perron theorem \cite[Theorem 3.2]{HKLY} to $M^{16}$,
 which  is the first power of $M$ that becomes a positive matrix.
\end{remark}
\section{Quantum cohomology of $Bl_{\bP^2}Q^4$}
In this section, we will study the quantum cohomology of the blow-up $X_2=Bl_{\rm \mathbb{P}^2}Q^4$ of $Q^4$ along a projective plane $\bP^2$ inside it, and verify Conjecture $\mathcal{O}$ and Conjecture \ref{conj} for $X_2$.

\subsection{Classical cohomology}
A geometric construction of $X_2$ is similar to that of $X_0$. Precisely, we let $\bP_+^2$ and $\bP^2_-$ be two disjoint $2$-planes in $Q^4$ such that $P.D.[\bP^2_\pm]=\wp_\pm\in H^4(Q^4)$. Then they represent different rulings by $2$-planes of $Q^4$, and
\[
N_{\bP^2_\pm|Q^4}\cong\Omega_{\bP_\pm^2}(2),
\]
where $\Omega$ denotes the cotangent sheaf.

   Consider the rational map $f': \bP^5\dashrightarrow \bP^2_-$, given by the projection from $\bP_+^2$, which  defines a graph $\Gamma_{f'}\subset\bP^5\times \bP^2_-$ by the Zariski closure of $\{(y, f'(y))\mid y\in \bP^5\setminus\bP^2_+\}$.
By abuse of notation, we consider the natural projections $\Gamma_{f'}\xrightarrow{p_1}\bP^5$ and $\Gamma_{f'}\xrightarrow{p_2}\bP^2_-$ as well as their restrictions
  $\pi_i:=p_i|_{X_2}$ with $X_2$ being the strict transform of $Q^4$ (i.e. the Zariski closure of $p_1^{-1}(Q^4\setminus \bP^2_+)$). Then  $\pi_1: X_2\to Q^4$  is the blow-up of $Q^4$ along $\bP_+^2$, and $\pi_2: X_2\to\bP_-^2$ endows $X_2$ with the projective bundle structure $\bP_{\bP_-^2}(\Omega_{\bP_-^2}(2)\oplus\cO_{\bP_-^2})$ \cite{SzWi}, which can be viewed as the projective compactification of
\(
Q^4\setminus\bP_+^2=N_{\bP^2_-|Q^4}\cong\Omega_{\bP_-^2}(2).
\)
The divisor  $E_2:=\bP_{\bP_-^2}(\Omega_{\bP_-^2}(2)\oplus\{0\})\subset\bP_{\bP_-^2}(\Omega_{\bP_-^2}(2)\oplus\cO_{\bP_-^2})$ is the exceptional divisor of the blow-up $X_2\xrightarrow{\pi_1}Q^4$.

The integral cohomology $H^*(X_2, \mathbb{Z})$ has a basis $\mathcal B_2$ arising from the blow-up  $X_2\xrightarrow{\pi_1}Q^4$, given by
\[
\mathcal{B}_2=\{\mathbf{1},H,\wp_+,\wp_-,\frac12H^3,\frac12H^4,E_2,S_2, L_2\};
\]
its dual basis $\mathcal{B}_2^\vee$ of $H^*(X_2, \mathbb{Z})$ with respect to the Poincar\'e pairing is given by
\[
\mathcal{B}_2^\vee=\{\frac12H^4,\frac12H^3,\wp_+,\wp_-,H,\mathbf{1},-L_2,-S_2,-E_2\}.
\]
Here $E_2\in H^2(X_2, \mathbb{Z})$ denotes the (cohomology) class of the exceptional divisor $E_2$ by abuse of notation,   $L_2$ is the  class of  a fiber in $E_2=\bP_{\bP_-^2}(\Omega_{\bP_-^2}(2)\oplus\{0\})$, and $S_2$ is the class of the preimage of a line in $\bP^2_+$ of $\pi_1$. 
By direct calculations, the non-zero cup products in $\mathcal B_2\setminus\{\mathbf{1}\}$ are given as follows.
\begin{align*}
H\cup H&=\wp_++\wp_-, &\wp_+\cup E_2&=L_2,\\
H\cup\wp_\pm&=\frac12H^3, & E_2\cup E_2&=-\wp_++S_2,\\
H\cup\frac12H^3&=\frac12H^4, &E_2\cup S_2&=-\frac12H^3+L_2,\\
H\cup E_2&=S_2, &E_2\cup L_2&=-\frac12H^4,\\
H\cup S_2&=L_2, &S_2\cup S_2&=-\frac12H^4.
\end{align*}
 Note that $X_2$ admits the projective bundle structure $X_2=\bP_{\bP_-^2}(\Omega_{\bP_-^2}(2)\oplus\cO_{\bP_-^2})\xrightarrow{\pi_2}\bP_-^2$. Hence, $H^*(X_2, \mathbb{Z})$ is generated by $H^2(X_2, \mathbb{Z})$. By the Leray-Hirsch theorem, we see that $H^*(X_2, \mathbb{Z})$ has another $\bZ$-basis $\mathcal B_2'$ given by
\[
\mathcal{B}_2'=\{\mathbf1,H_-,H_-^2,E_2,E_2H_-,E_2H_-^2,E_2^2,E_2^2H_-,E_2^2H_-^2\}.
\]
Here $H_-\in H^2(X_2,\mathbb{Z})$ is the pullback of the hyperplane class in $H^2(\bP^2_-, \mathbb{Z})$ via $\pi_2^*$.
 Let $\tilde f$ be the homology class of a fiber of $\pi_2$. Then $\pi_2$ is the contraction corresponding to the extremal ray $\bR_{\geq0}\tilde f$, and
\(
\tilde f=\ell-e.
\)
So we have
\[
H.\tilde f=E_2.\tilde f=1.
\]
Since $H_-.\tilde f=0$, it follows that we can write
\(
H_-=a(H-E_2)\mbox{ for some }a\in\bZ.
\)
Moreover, take a line in $\bP_-^2\subset Q^4$. Then the intersection number of $(\pi_1)_*H_-$ and this line is one. So we see that
\(
(\pi_1)_*H_-.\ell=1,
\)
which implies that
\(
H_-=H-E_2.
\)
In sum, we have
\[
\left\{
\begin{array}{cl}
\mathbf1&=\mathbf{1},\\
H_-&=H-E_2,\\
H_-^2&=\wp_--S_2,\\
E_2&=E_2,\\
E_2H_-&=\wp_+,\\
E_2H_-^2&=\frac12H^3-L_2,\\
E^2_2&=S_2-\wp_+,\\
E^2_2H_-&=L_2,\\
E^2_2H_-^2&=\frac12H^4,
\end{array}
\right.\mbox{ and }
\left\{
\begin{array}{cl}
\mathbf{1}&=\mathbf1,\\
H&=H_-+E_2,\\
\wp_+&=E_2H_-,\\
\wp_-&=H_-^2+E_2H_-+E_2^2,\\
\frac12H^3&=E_2H_-^2+E^2_2H_-,\\
\frac12H^4&=E^2_2H_-^2,\\
E_2&=E_2,\\
S_2&=E_2H_-+E^2_2,\\
L_2&=E^2_2H_-.
\end{array}
\right.
\]

\subsection{Two-point invariants}
Similar to the case of $X_0$, we study genus zero, two-point Gromov-Witten invariants of $X_2$, and obtain the following lemmas.

\begin{lemma}\label{lemma-blowupofQ4alongP2}
	The only non-zero, degree-$ke$ with $k\geq1$, two-point invariants with insertions in $\mathcal{B}_2$ are
	\[\<E_2,L_2\>_e^{X_2}=\<S_2,S_2\>_e^{X_2}=1.\]
\end{lemma}
\begin{proof}
The argument is completely similar to that of Lemma \ref{lemma-anotherblowupofQ4atonepoint}.
\end{proof}

\begin{lemma}
A curve class in $H_2(X_2, \mathbb{Z})$ admits non-zero two-point invariants only if it belongs to
	\(\{e,\ell-e,\ell,\ell+e\}.\)
\end{lemma}
\begin{proof}
An effective curve class has the form
\[
a(\ell-e)+be,\quad a,b\in\bZ_{\geq0}.
\]
Notice $c_1(X_2)=4H-E_2$ (see e.g. \cite[Chapter 4, Section 6]{GrHa}). Using the dimension constraint, we see that $a(\ell-e)+be$ admits non-zero two-point invariants only if
\[
3a+b\leq5\textrm{ and }(a,b)\neq(0,0).
\]
Note that $(a,b)\neq(0,0)$ since the space $\overline M_{0,2}(X_0,0)$ is empty. From Lemma \ref{lemma-blowupofQ4alongP2}, we exclude the cases $be$ for $2\leq b\leq5$.
	\end{proof}

\begin{lemma}
	The non-zero, degree-$(\ell-e)$, two-point invariant with insertions in $\mathcal{B}_2$ are
	\begin{align*}
	&\<\wp_-,\frac12H^4\>^{X_2}_{\ell-e}=\<E_2^2,\frac12H^4\>^{X_2}_{\ell-e}=1,\\
	&\<\frac12H^3,\frac12H^3\>^{X_2}_{\ell-e}=\<\frac12H^3,L_2\>^{X_2}_{\ell-e}=\<L_2, L_2\>^{X_2}_{\ell-e}=1.
	\end{align*}
\end{lemma}
\begin{proof}
Note that $X_2$ admits a projective bundle structure $X_2=\bP_{\bP_-^2}(\Omega_{\bP_-^2}(2)\oplus\cO_{\bP_-^2})\xrightarrow{\pi_2}\bP_-^2$, and $\ell-e$ is the homology class of a line in a fiber of $\pi_2$. By using similar arguments as in the proof of Lemma \ref{lemma-anotherblowupofQ4atonepoint}, one can prove that the only non-zero, degree-$(\ell-e)$, two-point invariant with insertions in $\mathcal{B}'_2$ are
\[
\<E^2_2,E^2_2H_-^2\>^{X_2}_{\ell-e}=\<E^2_2H_-,E_2^2H_-\>^{X_2}_{\ell-e}=1.
\]
So we can use this to determine invariants with insertions in $\mathcal B_2$. For example,
\[
\<\wp_-,\frac12H^4\>^{X_2}_{\ell-e}=\<H_-^2+E_2H_-+E_2^2,E_2^2H_-^2\>^{X_2}_{\ell-e}=1.
\]
We leave the rest cases to interested readers.
\end{proof}

\begin{lemma}
	The only non-zero, degree-$\ell$, two-point invariant with insertions in $\mathcal{B}_2$ is
	\[
	\<\frac12H^3,\frac12H^4\>^{X_2}_{\ell}=1.
	\]
\end{lemma}
\begin{proof}
By the dimension constraint, we only need to consider the case
\[
\<\frac12H^3,\frac12H^4\>_{\ell}^{X_2}\textrm{ and }\<L_2,\frac12H^4\>_{\ell}^{X_2}.
\]
For the former case, we use the blow-up formula \cite{Hu2}. For the latter case, observe that for $x\in\bP_+^2$ and $y\in Q^4$, there is a line in $Q^4$ passing through $x$ and $y$ iff $y\in H_x\cap Q^4$. As a consequence, for $x\in\bP_+^2$ and $y\in Q^4$ both in general position, there is no line in $Q^4$ passing through $x$ and $y$. So we have $\<L_2,\frac12H^4\>_{\ell}^{X_2}=0$.
\end{proof}

\begin{lemma}
	There is no non-zero, degree-$(\ell+e)$, two-point invariant with insertions in $\mathcal{B}_2$.
	\end{lemma}
\begin{proof}
By the dimension constraint, we only need to consider the case
\[
\<\frac12H^4,\frac12H^4\>_{\ell+e}^{X_2}.
\]
Observe that for $x, y\in Q^4$, there is a line in $Q^4$ passing through $x$ and $y$ iff $y\in H_x\cap Q^4$. As a consequence, for $x, y\in Q^4$ in general position, there is no line in $Q^4$ passing through $x$ and $y$. So we have $\<\frac12H^4,\frac12H^4\>_{\ell+e}^{X_2}=0$.
\end{proof}

Now we consider the operator $\hat c_1(X_2)$, where we recall $c_1(X_2)=4H-E_2$.
\begin{thm}\label{thmQ4P2}
    Both  Conjecture \ref{conj} and
  Conjecture $\mathcal{O}$ hold for the blow-up  $X_2=Bl_{\mathbb{P}^2} {Q}^4$.
\end{thm}
\begin{proof}
 Instead of the bases $\mathcal{B}_2$ and $\mathcal{B}_2'$, we consider another basis of  $H^*(X_2,\mathbb{Z})$ given by
$$\hat{\mathcal{B}}_2:=\{\mathbf{1}, H, H-E_2, \wp_+, \wp_-, \wp_--S_2, {1\over 2}H^3, {1\over 2}H^3-L_2, {1\over 2}H^4\}. $$
The associated matrix $M$ of the operator $\hat c_1(X_2)$ is given by $\hat c_1(X_2)\hat{\mathcal{B}}_2=\hat{\mathcal{B}}_2 M$ with
 $$M=\left(
    \begin{array}{ccccccccc}
0 & 0 & 0 & 0 & 3 & 0 & 4 & 4 & 0 \\ 3 & 0 & 1 & 0 & 0 & 0 & 0 & 0 & 4 \\ 1 & 0 & -1 & 0 & 0 & 0 & 3 & 0 & 0 \\ 0 & 4 & 3 & 0 & 0 &  0 & 0 & 0 & 0 \\ 0 & 3 & 0 & 0 & 0 & 1 & 0 & 0 & 0 \\
  0 & 1 & 4 & 0 & 0 & - 1 & 0 & 0 & 3 \\ 0 & 0 & 0 & 3 & 4 & 0 & 0 & 1 & 0 \\ 0 & 0 & 0 & 1 & 0 & 3 & 0 & - 1 & 0 \\ 0 & 0 & 0 & 0 & 0 & 0 & 4 & 3 & 0    \end{array}
  \right).
$$
\noindent Here the matrix $M$ is obtained by combining all the lemmas in this subsection, the divisor axiom for Gromov-Witten invariants, and the definition of the quantum product.

By direct calculations, the power $M^{13}$ is a positive matrix. Thus by Perron's theorem on positive matrices \cite{Perr}, the spectral radius $\rho(M^{13})$ is a simple eigenvalue of $M^{13}$, and the modulus of any other eigenvalue of $M^{13}$ is strictly less than $\rho(M^{13})$. Since $M$ is a real matrix and $13$ is odd,  it follows that    $\rho(M)=(\rho(M^{13}))^{1\over 13}$ is a simple eigenvalue of $M$, and the modulus of any other eigenvalue of $M$ is strictly less than $(\rho(M^{13}))^{1\over 13}$. Hence, Conjecture $\mathcal{O}$ holds for $X_2$. (Here we notice the Fano index of $X_2$ is equal to $1$, since
 $c_1(X_2).e=(4H-E_2).e=1$.)

Notice that $\det(\lambda I_{9}-M)|_{\lambda=4\sqrt{2}}<0$. Hence, there exists a real eigenvalue $\lambda_0$ in the interval $(4\sqrt{2}, +\infty)$. Hence, by Proposition \ref{proprhoQ}, we have
 $$\rho(\hat c_1(X_2))=\rho(M)\geq \lambda_0>4\sqrt{2}=\rho(\hat c_1(Q^4)).$$
That is, Conjecture \ref{conj} holds for $Y=Q^4$ and $Z=\mathbb{P}^2$.
\end{proof}
\begin{remark}  $M^{12}$
   is the first power of $M$ that becomes a positive matrix.
\end{remark}

\section*{Acknowledgements}

The authors would like to thank Kwokwai Chan  and Heng Xie   for  useful discussions.
 Hu is partially supported by NSFC Grants  11890662 and 11831017. Ke is partially supported by NSFC Grants 12271532 and  11831017. Li  is partially supported by NSFC Grant   11831017 and Guangdong Introducing Innovative and Enterpreneurial Teams No. 2017ZT07X355. Song is partially supported by Guangdong Basic and Applied Basic Research Foundation No. 2020A1515010876.

\end{document}